\documentclass{amsart}
\newcommand{\RN}[1]{\textup{\uppercase\expandafter{\romannumeral#1}}}
\usepackage{amsmath}
\usepackage{amssymb}
\usepackage{enumerate}
\usepackage[hidelinks]{hyperref}

\usepackage{mathtools}
\usepackage[capitalise]{cleveref}
\numberwithin{equation}{section}

\newcommand\mc{\mathcal}

\newcommand\mb{\mathbb}
\crefname{equation}{}{}
\newtheorem{theorem}{Theorem}[section]
\newtheorem{lemma}[theorem]{Lemma}

\newtheorem{corollary}[theorem]{Corollary}

\theoremstyle{definition}

\newtheorem{definition}[theorem]{Definition}
\newtheorem{example}[theorem]{Example}

\theoremstyle{remark}
\newtheorem{remark}[theorem]{Remark}
\usepackage{epigraph}
\usepackage{comment}
\usepackage{tikz}
\usepackage{tikz-cd}

\usepackage{enumitem}
\title{Resolution of the diagonal on the Root Stacks}
\author{Yu Zhao}
\begin{document}

\maketitle
\begin{abstract}
  In this paper, we give a new constructive proof of the semi-orthogonal decomposition of the derived category of (quasi)-coherent sheaves of root stacks, through an explicit resolution of the diagonal.
\end{abstract}
\setlength{\epigraphwidth}{0.45\textwidth}
\epigraph{\itshape To laugh is to live profoundly. }{\textit{Milan Kundera}}
\setlength{\epigraphwidth}{0.45\textwidth}
\section{Introduction}

\subsection{The main result of this paper}
The main result of this paper is a constructive proof on the semi-orthogonal decomposition of derived category of (quasi)-coherent sheaves on root stacks, which has been intensively studied by Ishii-Ueda \cite{MR3436544}, Bergh-Lunts-Schnürer \cite{MR3573964}, Kuznetsov-Perry \cite{MR3595897}, Bergh-Schnürer \cite{MR4031114}  and recently by Bodzenta-Donovan \cite{bodzenta2023root}:
\begin{theorem}
  \label{thm:main}
  Let $\mc{D}$ be an effective Cartier divisor of an algebraic stack $\mc{X}$. Given an integer $l>1$, let $\mc{X}_{\mc{D},l}$ as the $l$-th root stack of $\mc{X}$ along $\mc{D}$ and $\mc{D}_{l}:=\mc{D}\times_{B\mb{G}_{m}}B\mb{G}_{m}$, where the map from $\mc{D}\to B\mb{G}_{m}$ is decided by the line bundle $\mc{L}_{\mc{D},l}:=\mc{O}_{\mc{X}}(-\mc{D})|_{\mc{D}}$ and the map $B\mb{G}_{m}\to B\mb{G}_{m}$ is the $l$-th power map. Then under the following natural commutative diagram
\begin{equation*}
  \begin{tikzcd}
    \mc{D}_{l}\ar{r}{i_{\mc{X}_{\mc{D},l}}} \ar{d}{Bt_{\mc{D}}^{l}} & \mc{X}_{\mc{D},l}\ar{d}{\theta_{\mc{X}}^{l}}\\
    \mc{D}\ar{r}{i_{\mc{X}}} & \mc{X}.
  \end{tikzcd}
\end{equation*}
the following functors
\begin{gather*}
  \triangleright_{i,\mc{X}}:=(-\otimes\mc{L}_{D,l}^{i})\circ Ri_{\mc{X}_{\mc{D},l}*}\circ LBt_{\mc{D}}^{l*}:D^{+}_{qcoh}(\mc{D})\to D_{qcoh}^{+}(\mc{X}_{\mc{D},l}),\\
  L\theta_{\mc{X}}^{l*}:D_{qcoh}^{+}(\mc{X})\to D_{qcoh}^{+}(\mc{X}_{\mc{D},l})
\end{gather*}
are fully faithful. Moreover, let $D_{l,\mc{D}}^{i}$ be the image of $\triangleright_{i,\mc{X}}$, then for any $0\leq i\leq l-1$, we have the semi-orthogonal decomposition
\begin{equation}
  D^{+}_{qcoh}(\mc{X}_{\mc{D},l}):=<D_{l,\mc{D}}^{i-l+1},\cdots, D_{l,\mc{D}}^{-1}, L\theta^{*}_{l}D^{+}_{qcoh}(\mc{X}), D_{l,\mc{D}}^{0},\cdots, D_{l,\mc{D}}^{i-1}>.
\end{equation}
Similar arguments also hold for complexes with bounded or coherent cohomologies.                                                               
\end{theorem}
\begin{remark}
  Our setting is similar to Bergh-Lunts-Schnürer \cite{MR3573964}, which only assumes that $\mc{X}$ is an algebraic stack over $\mb{Z}$. Our strategy could also be generalized to the derived category of perfect complexes with very mild modifications. We will leave it as an exercise for readers with interest.
\end{remark}

\begin{remark}
  Bodzenta-Donovan \cite{bodzenta2023root} pointed out that the above semi-orthogonal decomposition is $2l$-periodic, and thus induces higher spherical functors. We will try to sketch a potential relation of those higher spherical functors with the categorical representation theory in \cref{sec1.4}.
\end{remark}

\begin{remark}
  The root stack construction by Cadman \cite{MR2306040} or Abramovich-Graber-Vistoli \cite{MR2450211} only requires a global section of a line bundle $\mc{L}$ over $\mc{X}$. In this paper, we only consider the case that the global section is injective i.e. induced by an effective Cartier divisor. We make this restriction so we can work in a more classical framework rather than the derived algebraic geometry, and this issue can be solved by introducing the definition ``virtual effective Cartier divisor'' in Khan-Rydh \cite{khan2018virtual}.
\end{remark}
\subsection{Resolutions of the diagonal and the semi-orthogonal decomposition}
Given an integer $n>0$, Beilinson \cite{beilinson1978coherent} gave a resolution of the diagonal of $\mb{P}^{n}$ by the long-exact sequence:
\begin{equation*}
0\to  \wedge^{n}\Omega_{\mb{P}^{n}}(n)\boxtimes\mc{O}_{\mb{P}^{n}}(-n)\to \cdots \to \Omega_{\mb{P}^{n}}\boxtimes \mc{O}_{\mb{P}^{n}}(-1)\to \mc{O}_{\mb{P}^{n}\times\mb{P}^{n}}\to \Delta_{*}\mc{O}_{\mb{P}^{n}}\to 0,
\end{equation*}
which induces the semi-orthogonal decomposition
\begin{equation*}
  D^{b}_{coh}(\mb{P}^{n})=<\mc{O}_{\mb{P}_{n}},\mc{O}_{\mb{P}_{n}}(1),\cdots, \mc{O}_{\mb{P}_{m}}(n)>=<\mc{O}_{\mb{P}_{n}},\Omega_{\mb{P}_{n}}(1),\cdots, \wedge^{n}\Omega_{\mb{P}_{n}}(n)>.
\end{equation*}

We refer to Kuznetsov \cite{MR3728631}\cite{kuznetsov2021semiorthogonal} for a general introduction to the semi-orthogonal decomposition of the derived category of coherent sheaves on algebraic varieties. One of the main purposes of this paper is to give an explicit resolution of the diagonal which induced the semi-orthogonal decomposition, following Appendix A of our work \cite{zhao2023generalized} for Orlov's semi-orthogonal decomposition theorem for blow-up of smooth varieties \cite{MR1208153}. We prove that
\begin{theorem}[\cref{main2}]
  \label{thm:main2}
  Let
  \begin{equation*}
    \triangleleft_{i,\mc{X}}:D_{qcoh}^{+}(\mc{X}_{\mc{D},l})\to D_{qcoh}^{+}(\mc{D})
  \end{equation*}
  be the cohomological degree $1$ shift of the right adjoint functor of $\triangleright_{i,\mc{X}}$. Given $0\leq n\leq m\leq n-1$, there exists 
\begin{equation*}
  \tau_{n,m,\mc{X}}:= D^{b}_{coh}(\mc{X}_{\mc{D},l}\times_{\theta_{\mc{X}}^{l},\mc{X},\theta_{\mc{X}}^{l}}\mc{X}_{\mc{D},l}).
\end{equation*}
such that the Fourier-Mukai functors generated by kernels $\tau_{n,m,\mc{X}}$ satisfy:
\begin{equation*}
\overline{\tau}_{n,m,\mc{X}}:D^{+}_{qcoh}(\mc{X}_{\mc{D},l})\to D^{+}_{qcoh}(\mc{X}_{\mc{D},l})
\end{equation*}
\begin{enumerate}
  \item we have
    \begin{equation*}
    \overline{\tau_{n,n,\mc{X}}}\cong \otimes \mc{O}_{\mc{X}_{\mc{D},l}}(-\mc{D}_{l})^{n},\quad \overline{\tau_{0,l-1,\mc{X}}}\cong L\theta_{\mc{X}}^{l*}R\theta_{\mc{X}*}^{l}
    \end{equation*}
  \item for any $0\leq n<m \leq l-1$, we have canonical triangles:
    \begin{align*}
      \overline{\tau_{n,m-1,\mc{X}}}\to \overline{\tau_{n,m,\mc{X}}}\to \bigoplus_{i=0}^{m} \triangleright_{m-i,\mc{X}} \triangleleft_{i,\mc{X}}\\
      \overline{\tau_{n,m,\mc{X}}}\to \overline{\tau_{n+1,m,\mc{X}}}\to \bigoplus_{i=n+1}^{l-1} \triangleright_{i-1-n,\mc{X}}\triangleleft_{i,\mc{X}}.
    \end{align*}.
      \item All those functors map complexes with bounded cohomologies (resp. coherent cohomologies) to complexes with bounded cohomologies (resp. coherent cohomologies).
  \end{enumerate}
\end{theorem}
\subsection{Motivation and derived birational geometry}
\label{sec13}
In \cite{zhao2023generalized}, we found a surprising relation between the resolution of the diagonal with the birational geometry of derived schemes. Now we try to briefly introduce this insight without introducing the machinery of $\infty$-categories.

For the simplicity, we assume that $\mc{X}=[\mb{A}^{1}_{\mb{C}}/\mb{C}^{*}]$ and the Cartier divisor is $B\mb{C}^{*}$. In this situation $\mc{X}_{\mc{D},l}$ is also $[\mb{A}^{1}_{\mb{C}^{1}}/\mb{C}^{*}]$ and
\begin{equation*}
  \mc{X}_{\mc{D},l}\times_{\mc{X}}\mc{X}_{\mc{D},l}\cong [\{(x,y)\in\mb{A}^{2}_{\mb{C}}|\prod_{j\in \mb{Z}/l\mb{Z}}(x-e^{\frac{2\pi i}{l}\cdot j }y)=0\}/(\mb{C}^{*}\times \mb{Z}/l\mb{Z})]
\end{equation*}
which is $l$-lines that intersect at the original point, quotient the action of $\mb{C}^{*}$ by the scalar action on both $x$ and $y$ and $\mb{Z}/l\mb{Z}$ by multiplying the roots of unity on $y$.

Now we blow up the origin point $(0,0)$ in 
\begin{equation*}
 \alpha_{l}:=\{(x,y)\in\mb{A}^{2}_{\mb{C}}| \prod_{j\in \mb{Z}/l\mb{Z}}(x-e^{\frac{2\pi i}{l}\cdot j }y)=0\}.
\end{equation*}
In the classical setting, we will get $l$ lines which do not intersect with each other, which is exactly $\mc{X}_{\mc{D},l}$ after quotient $\mb{C}^{*}\times \mb{Z}/l\mb{Z}$, and the natural projection is the diagonal morphism. However, in the sense of derived setting of Hekking \cite{Hekking_2022}, we should consider a family of varieties
\begin{equation*}
  \mc{M}_{k}:=\{((x_{1},y_{1}),[x_{2},y_{2}])\in \mb{A}^{2}_{\mb{C}}\times \mb{P}^{1}_{\mb{C}}|x_{1}y_{2}=y_{1}x_{2},x_{1}^{l-k}x_{2}^{l}=y_{1}^{l-k}y_{2}^{l}\}, \quad 0\leq k\leq l.
\end{equation*}
Then $\mc{M}_{l}\cong \mb{A}^{1}_{\mb{C}}\times \mb{Z}/l\mb{Z}$ and for other $k<l$, $\mc{M}_{k}$ contains $\mb{P}^{1}_{\mb{C}}$ which contains points such that $x_{1}=y_{1}=0$. In the sense of derived blow-up of Hekking \cite{Hekking_2022}, we will have
\begin{equation*}
  \mb{R}Bl_{\mc{M}_{k}}\mb{P}^{1}\cong \mc{M}_{k+1}, \quad \mb{R}Bl_{\alpha_{k}}\{(0,0)\}\cong \mc{M}_{1}
\end{equation*}
Moreover, in \cite{zhao2023generalized} we systematically studied the variation of certain quasi-coherent sheaves after the derived blow-up. Thus \cref{thm:main} actually is a direct corollary of the generalized vanishing theorem in \cite{zhao2023generalized}.

For algebraic stacks over $\mb{Z}$, we should be careful as the group of $l$-th of unity $\mu_{l}$ and the constant group $\mb{Z}/l\mb{Z}$ are not isomorphic in this situation. A detailed computation in this case is given in \cref{ext3}.

\subsection{Relations with the categorical representation theory}
\label{sec1.4}
In this subsection, we describe a potential relation between the semi-orthogonal decomposition of root stacks with the categorical representation, which we will study the detail in future work.

We denote $\Theta:=[\mb{A}^{1}/\mb{G}_{m}]$ and the morphism of $l$-th power on $\mb{A}^{1}$ induces a morphism $\theta_{l}:\Theta\to \Theta$. Then the composition of relative Fourier-Mukai transforms induce a monoidal structure on
\begin{equation*}
  D^{b}_{coh}(\Theta\times_{\theta_{l},\Theta,\theta_{l}}\Theta).
\end{equation*}
Then for any $l$-the root stack $\mc{X}_{\mc{D},l}$, the pull-back morphisms induce that $D^{b}_{Coh}(\mc{X}_{\mc{D},l})$ is a representation of the monoidal category $D^{b}_{coh}(\Theta\times_{\theta_{l},\Theta,\theta_{l}}\Theta)$. Moreover, we notice that $Coh(\Theta\times_{\theta_{l},\Theta,\theta_{l}}\Theta)$ consists of $\mb{Z}[x]\otimes_{\mb{Z}[x^{l}]}\mb{Z}[x]$ bimodules with certain grading. We expect that the semi-orthogonal decomposition of the root stacks follows directly from the categorical representation theory of those bimodules.
\subsection{Organization of this paper}
As we will work on stacks over $\mathrm{Spec}(\mb{Z})$, we review the background of affine groups over commutative rings and the derived category of (quasi)-coherent sheaves on algebraic stacks in \cref{sec2} and \cref{sec3} respectively. Experts or readers who only care about schemes over $\mb{C}$ should feel free to just  read \cref{ext3}, \cref{universal} and \cref{der}. The proof of \cref{thm:main} and \cref{thm:main2} is given in \cref{sec4}.
\subsection{Acknowledgments}
The paper is directly inspired by Bodzenta-Donovan \cite{bodzenta2023root}, where the author learned and discussed with them in the conference ``Current Trends in Categorically Approach to Algebraic and Symplectic Geometry \RN{2}'' in June 2023 at Kavli IPMU. The author would like to appreciate Agnieszka Bodzenta and Will Donovan for their interest and many helpful discussions in this topic.

Part of this paper was given in my lecture series in the Shanghai Tech University and Chinese Academy of Sciences. The author would like to thank the above institutes and Zhiyuan Ding, Mingliang Cai, Ziyu Zhang, Siqi He and Baohua Fu for their invitations and support.

The author is supported by World Premier International Research Center Initiative (WPI initiative), MEXT, Japan, and Grant-in-Aid for Scientific Research grant  (No. 22K13889) from JSPS Kakenhi, Japan.

\section{Affine group schemes and actions}
\label{sec2}

In this section, we always assume that $R$ is a commutative ring. We identify the category of affine schemes over $\mathrm{Spec}(R)$ with the opposite category of commutative $R$-algebras. Given an affine $R$-scheme $X:=R[X]$, we identify the abelian category of quasi-coherent (resp. coherent if $R[X]$ is Noetherian) sheaves $\mathrm{QCoh}(X)$ (resp. $
\mathrm{Coh}(X)$) with the abelian category of (resp. finite generated) $R[X]$-modules $R[X]-\mathrm{Mod}$ (resp. $R[X]-\mathrm{Mod}^{fg}$).

\subsection{Affine group schemes} 

\begin{definition}
  We define an $R$-affine group $G$ as a flat, finite type and commutative $R$-algebra $\pi_{G}:R\to R[G]$ with $R$-algebra morphisms
  \begin{equation*}
    m_{G}:R[G]\to R[G]\otimes_{R}R[G], \quad i_{G}:R[G]\to R[G], \quad e_{G}:R[G]\to R
\end{equation*}
  which we call multiplication, inversion, and identification respectively, such that they satisfy the group laws:
     \begin{align*}
     & (m_{G}\otimes_{R}id_{R[G]})\circ m_{G}= (id_{R[G]}\otimes_{R}m_{G})\circ m_{G}, & R[G]\to R[G]\otimes_{R}R[G]\otimes_{R}R[G], \\
     & id_{R[G]}= (id_{R[G]}\otimes_{R}e_{G})\circ m_{G}= (e_{G}\otimes_{R} id_{R[G]})\circ m_{G}, & R[G]\to R[G], \\
     & \pi_{G}\circ e_{G}= (i_{G}\otimes_{R}id_{R[G]})\circ m_{G}=(id_{R[G]}\otimes_{R}i_{G})\circ m_{G}, &R[G]\to R[G].
  \end{align*}

  We define an $R$-affine group morphism from an $R$-affine group $G$ to another $R$-affine group $H$ as a morphism of $R$-algebras $f:R[H]\to R[G]$ such that
  \begin{align*}
    (f\otimes_{R}f) \circ m_{H}=m_{G}\circ f: R[H]\to R[G]\otimes_{R}R[G]
  \end{align*}
\end{definition}

The $R$-affine group has the following properties by definition:
\begin{itemize}
\item   Given a morphism of commutative ring $S\to T$ and an $R$-affine group $G$, $G_{T}:=\mathrm{Spec}(R[G]\otimes_{S}T)$ is also an $T$-affine group by the induced morphisms. Similarly, affine group morphisms are also stable under base change.
\item   Given two morphisms of $R$-affine groups
  \begin{equation*}
    f_{1}:G_{1}\to G, \quad f_{2}:G_{2}\to G,
  \end{equation*}
  $G_{1}\times_{f_{1},G,f_{2}}G_{2}$ is also an affine $R$-group with the multiplication induced by $m_{G_{1}}$ and $m_{G_{2}}$ (if it is fppf over $\mathrm{Spec}(R)$). Particularly, given a morphism of $R$-affine group $f:G\to H$, we denote $ker(f):=G\times_{f,H,\pi_{H}}\mathrm{Spec}(R)$.
\end{itemize}

\begin{example}
The scheme $\mb{G}_{m}:=\mathrm{Spec}(\mb{Z}[t,t^{-1}])$ has a $\mb{Z}$-affine group structure
   \begin{align*}
&m_{\mb{G}_{m}}:\mb{Z}[t,t^{-1}]\to \mb{Z}[t_{1},t^{-1}_{1},t_{2},t_{2}^{-1}] , &\quad t\to t_{1}t_{2}, \\
 &i_{\mb{G}_{m}}:\mb{Z}[t,t^{-1}]\to \mb{Z}[t,t^{-1}],&\quad t\to t^{-1}, \\
 &e_{\mb{G}_{m}}:\mb{Z}[t,t^{-1}]\to \mb{Z}, &\quad t\to 1.
   \end{align*}

   Given any integer $l$, we have a $\mb{Z}$-group $t^{l}:\mb{G}_{m}\to \mb{G}_{m}$ by
   \begin{equation*}
     t^{l}:\mb{Z}[t,t^{-1}]\to \mb{Z}[t,t^{-1}], \quad t\to t^{l}.
   \end{equation*}
   When $l>0$, we denote
   \begin{equation*}
     \mu_{l}:=ker(t^{l})=\mathrm{Spec}(\mb{Z}[t]/(t^{l-1})).
   \end{equation*}
 We denote $\mb{G}_{m,R}:=\mb{G}_{m}\times_{Z}R$ and make similar notations for $\mu_{l,R}$ and $t^{l}_{R}$.
\end{example}

\subsection{Representations of an affine group}
\begin{definition}
  Let $G$ be an $R$-affine group. We define a $G$-representation $M$ as an $R$-module with an $R$-module morphism
  \begin{equation*}
    \sigma_{G,M}:M\to R[G]\otimes_{R}M
  \end{equation*}
  such that
   \begin{align*}
    &(m_{G}\otimes_{R}id_{M})\circ \sigma_{G,M}= (id_{R[G]}\otimes_{R}\sigma_{G,M})\circ \sigma_{G,M},& M\to R[G]\otimes_{R}R[G]\otimes_{R}M &\\
    & id_{M}=(e_{G}\otimes_{R}id_{M})\circ \sigma_{G,M}, & M\to M. &
   \end{align*}
   We say that $M$ is a finitely generated $G$-representation if $M$ is finitely generated as an $R$-module.  Given two $G$-representations $M$ and $N$, we define a $G$-equivariant morphism from $M$ to $N$ as an $R$-module morphism $f:M\to N$ such that
     \begin{equation*}
       \sigma_{G,N}\circ f=(id_{R[G]}\otimes_{R}f)\circ \sigma_{G,M}, \quad M\to R[G]\otimes_{R}N.
     \end{equation*}
   The category of $G$-representations with $G$-equivariant morphisms forms an abelian category, which we denote as $\mathrm{Rep}(G)$. Moreover, if $R$ is Noetherian, the category of finitely generated $G$-representations form an abelian subcategory of $\mathrm{Rep}(G)$, which we denote as $\mathrm{Rep}(G)^{fg}$.
\end{definition}
The category of $G$-representations has the following properties by definition:
\begin{enumerate}
  \item The tensor product of $R$-modules induces a canonical monoidal structure for $\mathrm{Rep}(G)$.
\item Given a $R$-affine group morphism $\psi:G\to H$, let $M$ be an $H$-representation. Then
  \begin{equation*}
    \sigma_{G,M}:=\psi\otimes_{R}\sigma_{H,M}:M\to R[G]\otimes_{R}M 
  \end{equation*}
  induces a $G$-representation structure on $M$. Hence it induces an exact and monoidal pull-back functor:
  \begin{equation*}
    \psi^{*}:\mathrm{Rep}(H)\to \mathrm{Rep}(G).
  \end{equation*}
\end{enumerate}

\begin{example}
  \label{ex1}
   There exists a canonical monoidal equivalence between $\mathrm{Rep}(\mb{G}_{m,R})$ with the category of $\mb{Z}$-graded (resp. $\mb{Z}/l\mb{Z}$) graded $R$-modules $R-\mathrm{Mod}_{\mb{Z}}$ in the following way: given a $\mb{Z}$-graded $R$-module $M=\oplus_{d\in \mb{Z}}M_{d}$, we denote $M|_{d}:M\to M_{d}$ as the projection morphism. Then
  \begin{equation*}
    \sigma_{\mb{G}_{m,R},M}:=\bigoplus_{d\in \mb{Z}}t^{d}\otimes_{R}M|_{d}:M\to M[t,t^{-1}]
  \end{equation*}
  is a $\mb{G}_{m,R}$-representation structure on $M$. Conversely, given $\mb{G}_{m,R}$ representation $M$, the $R$-module morphism
  \begin{equation*}
    \sigma_{G,M}:M\to M[t,t^{-1}]
  \end{equation*}
  could be written as
  \begin{equation*}
    \sum_{d\in \mb{Z}}t^{d}\otimes\sigma_{d}
  \end{equation*}
  where $\sigma_{d}$ is an $R$-module endmorphism of $M$ for each $d$. The law of group actions requires that
  \begin{equation*}
    \sigma_{d}\sigma_{e}=\delta_{de}\sigma_{e}, \quad \sum_{d\in \mb{Z}}\sigma_{d}=1,
  \end{equation*}
  where $\delta_{de}$ is the Kronecker symbol. Hence we have $M\cong \oplus_{d\in \mb{Z}}\sigma_{d}M$.

  Under the above equivalences, we compute the pull-back and push-forward functor of $t^{l}_{R}:\mb{G}_{m,R}\to \mb{G}_{m,R}$, which is represented by
  \begin{align*}
    t^{l*}_{R}:R-\mathrm{Mod}_{\mb{Z}}\to R-\mathrm{Mod}_{\mb{Z}}, \quad \bigoplus_{d\in \mb{Z}}M_{d}\to \bigoplus_{d\in l\mb{Z}}M_{d/l}, \\
    t^{l}_{R*}:R-\mathrm{Mod}_{\mb{Z}}\to R-\mathrm{Mod}_{\mb{Z}}, \quad \bigoplus_{d\in \mb{Z}}M_{d}\to \bigoplus_{d\in \mb{Z}}M_{dl}.
  \end{align*}
  Both $t_{R}^{l*}$ and $t^{l}_{R*}$ are exact functors and maps finite generated $R$-modules to finite generated $R$-modules.

  Given a rational number $i$, we denote
  \begin{equation*}
    \mc{L}^{i}_{R}:=
    \begin{cases}
      \bigoplus_{d=i}R, & i\in \mb{Z},\\
      0, & i\not\in \mb{Z}.
    \end{cases}
  \end{equation*}
Then we have
  \begin{equation*}
    t_{R}^{l*}(\mc{L}_{R}^{i})\cong \mc{L}_{R}^{li},\quad t_{R*}^{l}(\mc{L}_{R}^{i})\cong \mc{L}_{R}^{i/l}.
  \end{equation*}
  Like the case of $\mb{G}_{m,R}$, we have a canonical equivalence between the category of $\mu_{n,R}$ representations and the category of $\mb{Z}/l\mb{Z}$-graded $R$-modules.
  \begin{equation*}
    \mathrm{Rep}(\mu_{n,R})\cong R-\mathrm{Mod}_{\mb{Z}/l\mb{Z}}.
  \end{equation*}
\end{example}
\subsection{Equivariant affine schemes and morphisms}
\begin{definition} 
  Given an $R$-affine group $G$, we define a $G$-equivariant affine scheme as a $G$-representation $R[X]$ with a $G$-equivariant morphism
  \begin{equation*}
    *_{X}:R[X]\otimes_{R}R[X]\to R[X].
  \end{equation*}
  such that $R[X]$ is an commutative $R$-algebra under the multiplication $*_{X}$. We denote
  \begin{equation*}
    \sigma_{G,X}:=\sigma_{G,R[X]},\quad \pi_{G,X}:=\pi_{G,R[X]}.
  \end{equation*}

  We define an $R$-affine equivariant scheme as a pair $(G,X)$ where $G$ is an $R$-affine group and $X$ is a $G$-equivariant affine scheme. Given two affine equivariant schemes $(G,X)$ and $(H,Y)$, an equivariant morphism from $X$ to $Y$ is defined as a pair $(\psi,f)$, where $\psi:G\to H$ is an $R$-affine group morphism and $f:R[Y]\to R[X]$ is an $R$-algebra morphism such that 
  \begin{equation*}
  (\psi\otimes_{R}f)  \circ \sigma_{H,Y}= \sigma_{G,X}\circ g, \quad R[Y]\to R[G]\otimes_{R}R[X]
  \end{equation*}
\end{definition}
The category of equivariant affine schemes and morphisms has the following properties by definitions:
\begin{enumerate}
\item Every $R$-affine group morphism $f:G\to H$ induces a $G$-equivariant structure on $H$. Particularly, there is a canonical $G$-equivariant structure on $\mathrm{Spec}(R)$.
\item Let $\psi:G\to H$ be an $R$-affine group morphism and $Y$ be a $H$-equivariant scheme, then
  \begin{equation*}
    \sigma_{G,X}:=(\psi\otimes_{R}id_{R[X]})\circ \sigma_{H,X}:R[X]\to R[G]\otimes_{R}R[X]
  \end{equation*}
  induces a $G$-equivariant structure on $X$. Moreover, any equivariant morphism of equivariant schemes:
  \begin{equation*}
    (G,X)\xrightarrow{(\psi,f)}(H,Y)
  \end{equation*}
  factors through
  \begin{equation*}
    (G,X)\xrightarrow{(id,f)} (G,Y)\xrightarrow{(\psi,id)} (H,Y).
  \end{equation*}
  \item Given two equivariant morphisms
  \begin{align*}
    (\psi_{1},f_{1}):G_{1}\times_{\mathrm{Spec}(R)}X_{1}\to G\times_{\mathrm{Spec}(R)}X,\\ (\psi_{2},f_{2}):G_{2}\times_{\mathrm{Spec}(R)}X_{2}\to G\times_{\mathrm{Spec}(R)}X,
  \end{align*}
  it induces a $G_{1}\times_{G}G_{2}$-equivariant structure on $X_{1}\times_{X}X_{2}$ by the actions $\sigma_{G_{1},X_{1}}$ and $\sigma_{G_{2},X_{2}}$. 
\end{enumerate}

\begin{example}[GIT quotients]\label{ex2}By \cref{ex1}, there is a canonical equivalence between the opposite category of $\mb{Z}$-graded (resp. $\mb{Z}/l\mb{Z}$) commutative $R$-algebras with the
  category of $\mb{G}_{m,R}$ (resp. $\mu_{l,R}$) equivariant affine schemes. Moreover, given an $\mb{Z}$-graded (resp. $\mb{Z}/l\mb{Z}$-graded) algebra
  \begin{equation*}
    R[X]=\bigoplus_{d\in \mb{Z}}R[X]_{d} \textit{ (resp. } \bigoplus_{d\in \mb{Z}/l\mb{Z}}R[X]_{d} \textit{)},
  \end{equation*}
 
  $R[X]_{0}$ is also an $R$-algebra and all $R[X]_{i}$ is an $R[X]_{0}$-module. The inclusion of $R[X]_{0}$ into $R[X]$ induces a canonical equivariant morphism
  \begin{equation*}
    \phi_{0,X}:(\mb{G}_{m},X)\to (\mathrm{Spec}(R),\mathrm{Spec}(R[X]_{0}))\textit{ or } \phi_{l,X}:(\mu_{l},X)\to (\mathrm{Spec}(R),R[X]_{0}).
  \end{equation*}
  
\end{example}

\subsection{Equivariant modules}
\begin{definition}
  Given an $R$-affine equivariant scheme $(G,X)$, we define a $G$-equivariant $R[X]$-module as a $G$-representation $M$ with a $G$-equivariant morphism
  \begin{equation*}
    *_{M}:R[X]\otimes_{R}M\to M
  \end{equation*}
  such that $*_{m}$ induces an $R[X]$-module structure on $M$.  Given two $G$-equivariant $R[X]$-modules $M,N$, we define a $G$-equivariant $R[X]$-module morphism from $M$ to $N$ as a
  $G$-equivariant morphism $f:M\to N$ such that $f\circ *_{M}=*_{N}\circ id_{R[X]}\otimes_{R}f$. We say that $M$ is finitely generated if it is finitely generated as $R[X]$-module.

  We denote $\mathrm{QCoh}_{G}(X)$ as the abelian category of $G$-equivariant $R[X]$-modules. If $R[X]$ is Noetherian, we denote $\mathrm{Coh}_{G}(X)$ as the abelian category of finite generated $G$-equivariant $R[X]$-modules.
   \end{definition}
 The category of equivariant modules has the following property by definition
 \begin{enumerate}
 \item We have $\mathrm{QCoh}_{G}(\mathrm{Spec}(R))=\mathrm{Rep}(G)$ and $\mathrm{Coh}_{G}(\mathrm{Spec}(R))=\mathrm{Rep}(G)^{fg}$.
 \item The category of $G$-equivariant $R[X]$-modules has a canonical monoidal structure: let $M$ and $N$ be two $G$-equivariant $R[X]$-modules, the $R[X]$-module $M\otimes_{R[X]}N$, regarding as an $R$-module, is the cokernel of
   \begin{equation*}
     R[X]\otimes_{R} M\otimes_{R}N\to M\otimes_{R}N, \quad x\otimes m\otimes n \to *_{M}(xm)\otimes n-x\otimes *_{N}(xn)
   \end{equation*}
   which has a canonical $G$-equivariant structure.
 \item Given an equivariant morphism $(\psi,f):(G,X)\to (H,Y)$, there is a canonical pull-back functor
   \begin{equation*}
     (\psi,f)^{*}:\mathrm{QCoh}_{H}(Y)\to \mathrm{QCoh}_{G}(X)
   \end{equation*}
   defined in the following way:
   \begin{itemize}
   \item when $\psi=id$, we define $(f,id)^{*}M:=M\otimes_{R[X]}R[Y]$;
   \item when $f=id$, we define $(f,id)^{*}M:=f^{*}M$ as a $G$-representation, and the $R[X]$-module structure is still induced by $*_{M}$.
   \item in general, we define $(\psi,f)^{*}:=(\psi,id)^{*}(id,f)^{*}$.
   \end{itemize}
   The pull-back functor $(\psi,f)^{*}$ maps finite generated $H$-equivariant $R[Y]$-modules to finite generated $G$-equivariant $R[X]$-modules. We denote
   \begin{equation*}
     (\psi,f)_{*}:\mathrm{QCoh}_{G}(X)\to \mathrm{QCoh}_{H}(X)
   \end{equation*}
   as the right adjoint functor of $(\psi,f)^{*}$ if it exists.
 \end{enumerate}
 \begin{example}
   \label{ex3}  \label{ext5}
   Let $R[X]$ be a $\mb{Z}$-graded $R$-algebra, with the $\mb{G}_{m,R}$ action induced by the grading. By \cref{ex1} and \cref{ex2}, there is a canonical equivalence between $\mathrm{QCoh}_{\mb{G}_{m,R}}(X)$ (resp. $\mathrm{Coh}_{\mb{G}_{m,,R}}(X)$ if $R[X]$ is Noetherian) with the category of $\mb{Z}$-graded (resp. finite generated) $R[X]$-modules. The similar arguments also work for $\mu_{l,R}$-equivariant schemes.

   Particularly, given a $\mb{Z}/l\mb{Z}$-graded $R$-algebra $R[X]$, under the GIT quotient map
   \begin{equation*}
     \phi_{l,X}:(\mu_{l},X)\to (\mathrm{Spec}(R),\mathrm{Spec}(R[X]_{0}))
   \end{equation*}
   the pull-back and push-forward maps are represented by
   \begin{align*}
     \phi_{l,X}^{*}:\mathrm{QCoh}(\mathrm{Spec}(R[X]_{0}))\to \mathrm{QCoh}_{\mu_{l,R}}(X):&\quad  M\to \bigoplus_{d\in \mb{Z}/l\mb{Z}}M\otimes_{R[X]_{0}}R[X]_{i},\\
     \phi_{l,X*}: \mathrm{QCoh}_{\mu_{l,R}}(X)\to \mathrm{QCoh}(\mathrm{Spec}(R[X]_{0})): &\quad  \bigoplus_{d\in \mb{Z}/l\mb{Z}}M_{i}\to M_{0}
   \end{align*}
   Particularly, $\phi_{0,X*}$ is always exact, and maps finite generated $R[X]$-graded modules to finite generated $R[X]_{0}$-modules if $R[X]$ is finite generated as a $R[X]_{0}$-module and $R[X]_{0}$ is Noetherian.
 \end{example}

\subsection{The $\mb{G}_{m,R}$-equivariant scheme $\mb{A}_{R}^{1}$}
\label{ext3}\label{ex4}
The scheme $\mb{A}^{1}_{R}:=\mathrm{Spec}(R[x])$ has a universal $\mb{G}_{m,R}$-equivariant structure by associating the degree of $x$ as $1$. Given a positive integer $l$, we denote $x^{l}:\mb{A}^{1}\to \mb{A}^{1}$ by
  \begin{equation*}
    \mathbb{Z}[x]\to \mathbb{Z}[x], \quad x\to x^{l}.
  \end{equation*}
  \begin{equation*}
    (t^{l},x^{l}):(\mb{G}_{m},\mb{A}^{1})\to (\mb{G}_{m}, \mb{A}^{1})
  \end{equation*}
  is an equivariant action. Moreover, we have
  
  \begin{equation*}
    (\mb{G}_{m},\mb{A}^{1})_{(t^{l},x^{l})}\times_{(\mb{G}_{m},\mb{A}^{1}),(t^{l},x^{l})}(\mb{G}_{m},\mb{A}^{1})\cong (\mb{G}_{m}\times \mu_{l}, \alpha_{l})
  \end{equation*}
  where $\alpha_{l}:=\mathrm{Spec}(\mb{Z}[x,y]/(x^{l}-y^{l}))$ and the group action is induced by
  \begin{equation*}
    (t,\mu)(x,y)\to (tx,t\mu y)
  \end{equation*}
  There are two equivariant morphisms:
  \begin{equation*}
    \alpha_{l,1},\alpha_{l,2}:(\mb{G}_{m}\times \mu_{l},\alpha_{l})\to (\mb{G}_{m},\mb{A}^{1})
  \end{equation*}
  such that
  \begin{equation*}
    \alpha_{l,1}((t,\mu),(x,y))=(t,x), \quad \alpha_{l,2}((t,\mu),(x,y))=(t\mu,y).
  \end{equation*}

  The affine group $\mb{G}_{m}\times \mu_{l}$ also acts on $\mb{A}^{1}\times \mu_{l}$ by $(t,\mu_{1})\circ (x,\mu_{2})\to (tx, \mu_{1}\mu_{2})$. We have an $\mb{G}_{m}\times \mu_{l}$-equivariant morphism of affine schemes from $\mb{A}^{1}\times \mu_{l}$ to $\alpha_{l}$ by
  \begin{equation*}
    \Delta_{t^{l},x^{l}}:\mb{Z}[x,y]/(x^{l}-y^{l})\to \mb{Z}[x,t]/(t^{l}-1), \textit{ by } x\to x, y\to ty
  \end{equation*}
  Here $x,y,t$ had homogeneous degree $(1,0)$, $(1,1)$ and $(0,1)$ with respect to the $\mb{G}_{m}\times \mu_{l}$ action.

  The morphism $\Delta_{t^{l},x^{l}}$ induces morphisms of equivariant $\mb{Z}[x,y]/(x^{l}-y^{l})$ modules
\begin{equation*}
  \Delta_{t^{l},x^{l}}^{n}:(x,y)^{n}\mb{Z}[x,y]/(x^{l}-y^{l})\to (x^{n})\mb{Z}[x,t]/(t^{l}-1), \quad n\in \mb{Z}_{\geq 0}
\end{equation*}
which is always injective and is an isomorphism when $n=l-1$. For integers  $0\leq m\leq n\leq l-1$, we define 
\begin{equation*}
  \tau_{m,n}:=\frac{(x^{n})\mb{Z}[x,t]/(t^{l}-1)\oplus (x,y)^{m}\mb{Z}[x,y]/(x^{l}-y^{l})}{(x,y)^{n}\mb{Z}[x,y]/(x^{l}-y^{l})}\in \mathrm{Coh}_{\mb{G}_{m}\times \mu_{l}}(\alpha_{l})
\end{equation*}
The modules $\tau_{m,n}$ have the following properties:
\begin{enumerate}
\item we have
  \begin{equation*}
    \tau_{n,n}\cong \Delta_{t^{l},x^{l}*}(x^{n}\mb{Z}[x,t]/(t^{l}-1)), \quad \tau_{n,l-1}\cong (x,y)^{n}\mb{Z}[x,y]/(x^{l}-y^{l})
  \end{equation*}
\item When $n< m\leq l-1$, there exists a canonical injection morphism
  \begin{equation*}
    \tau_{n,m-1}\to \tau_{n,m}, \quad \tau_{n,m}\to \tau_{n+1,m}
  \end{equation*}
  such that cokernels are
  \begin{equation*}
    \bigoplus_{i=0}^{m}\mb{Z}<m,i>,  \textit{ and } \bigoplus_{i=n+1}^{l-1}\mb{Z}<n+1,i>,
  \end{equation*}
  respectively, where $<a,b>$ is the homogeneous degree with respect to the $\mb{G}_{m}\times \mu_{l}$ action.
\end{enumerate}

\section{Derived category of (quasi)-coherent sheaves on algebraic stacks}
\label{sec3}
In this section, we review the background about algebraic stacks and (quasi)-coherent sheaves on algebraic stacks, following \cite[\href{https://stacks.math.columbia.edu/tag/0ELS}{Tag 0ELS}]{stacks-project}. Experts or readers who are only interested in the scheme cases should feel free to read from \cref{universal}.

\subsection{Quotient stacks} Given an $R$-equivariant scheme  $(G,X)$, the groupoid in schemes $(X,G\times_{\mathrm{Spec}(R)}X,\sigma_{G,X},\pi_{G}\times_{\mathrm{Spec}(R)}id_{X},m_{G}\times_{\mathrm{Spec}(R)}id_{X})$ induces an algebraic stack which we denote as $[X/G]$, following \cite[\href{https://stacks.math.columbia.edu/tag/044O}{Tag 044O}]{stacks-project}. Let
\begin{equation*}
\sigma_{X}^{G}:X\to [X/G]  
\end{equation*}
be the canonical quotient map, we have the Cartesian diagram:
\begin{equation*}
  \begin{tikzcd}[column sep=2cm]
    G\times_{\mathrm{Spec}(R)}X \ar{r}{\sigma_{G,X}}\ar{d}{\pi_{G,X}} & X\ar{d}{\sigma_{X}^{G}} \\
    X\ar{r}{\sigma_{X}^{G}} & {[X/G]}.
  \end{tikzcd}
\end{equation*}
where all the morphisms are fppf coverings.

Given an $R$-equivariant morphism  $(\psi,f):(G,X)\to (H,Y)$, the morphism of groupoid in $R$-schemes $(f,f\times_{\mathrm{Spec}(R)}f)$ induces a morphism of quotient stacks:
\begin{equation*}
  [f/\psi]:[X/G]\to [Y/H],
\end{equation*}
Particularly, given a morphism of $R$-affine groups $\psi:G\to H$, we denote
\begin{equation*}
  BG:=[\mathrm{Spec}(R)/G], \quad BH:=[\mathrm{Spec}(R)/H], \quad B\psi:=[id/\psi]:BG\to BH.
\end{equation*}

Given two morphisms $(\psi_{1},f_{1}):(G_{1},X_{1})\to (H,Y)$ and $(\psi_{2},f_{2}):(G_{2},X_{2})\to (H,Y)$, the fiber product of equivariant morphisms induces the fiber product of algebraic stacks if $G_{1}\times_{H}G_{2}$ is also fppf over $R$.
\begin{example}
  \label{ext4}
  We follow Halpern-Leistner \cite{halpern2020derived} to denote $\Theta:=[\mb{A}^{1}/\mb{G}_{m}]$ and denote
  \begin{equation*}
    \theta^{l}:=[x^{l}/t^{l}]:\Theta\to \Theta.
  \end{equation*}
  Then the diagonal of $\theta^{l}$, which we denote as $\Delta_{\theta^{l}}$ is represented by
  \begin{equation*}
  [\Delta_{t^{l},x^{l}}/id]:[\mb{A}^{1}\times \mu_{l}/\mb{G}_{m}\times\mu_{l}]\to [\alpha_{l}/\mb{G}_{m}\times \mu_{l}]
\end{equation*}
in  \cref{ex4}.
\end{example}

\subsection{Morphisms of algebraic stacks}

We refer to the Stacks Project  \cite[\href{https://stacks.math.columbia.edu/tag/04XM}{Tag 04XM}]{stacks-project}, for the property of being flat, locally of finite type, quasi-compact, quasi-separated and having an affine diagonal for morphisms of algebraic stacks. Moreover, we mention the following facts
\begin{enumerate}
\item Given a morphism $f:\mc{X}\to \mc{Y}$ such that both $\mc{X}$ and $\mc{Y}$ are represented by schemes, the definition of the above properties coincide with the definition for morphisms of schemes.
\item The above properties are stable under base change, i.e. given a Cartesian diagram of algebraic stacks
  \begin{equation}
    \label{eqn1}
      \begin{tikzcd}
        \mc{X}'\ar{d}{f'}\ar{r} & \mc{X} \ar{d}{f}\\
        \mc{Y}'\ar{r}{g} & \mc{Y},
      \end{tikzcd}\tag{*}
    \end{equation}
    $f'$ is flat (resp. locally of finite type, quasi-compact, quasi-separated, with an affine diagonal) if $f$ is flat (resp. locally of finite type, quasi-compact, quasi-separated, with an affine diagonal). Moreover, if $g$ is a smooth atlas, $f$ is flat (resp. locally of finite type, quasi-compact, quasi-separated, with an affine diagonal) if $f'$ is flat (resp. locally of finite type, quasi-compact, quasi-separated, with an affine diagonal).
  \item We consider a morphism of affine quotient stacks
    \begin{equation*}
      [f/\psi]:[X/G]\to [Y/H].
    \end{equation*}
    The morphism $[f/\psi]$ is always quasi-compact. It is flat (resp. of locally finite type) if $f$ is flat (resp. locally of finite type). It is quasi-separated and with an affine diagonal if $\psi$ is faithfully flat.
  \item Given a quasi-compact morphism of algebraic stacks $f:\mc{X}\to \mc{Y}$ such that $\mc{Y}$ is represented by an affine scheme, then $\mc{X}$ has a smooth cover by an affine scheme.
\end{enumerate}

\begin{example}
  The morphisms
  \begin{equation*}
    \theta^{l}:\Theta\to \Theta, \quad Bt^{l}:B\mb{G}_{m}\to B\mb{G}_{m}
  \end{equation*}
  are both locally of finite type, flat, quasi-compact, quasi-separated, and have affine diagonals.
\end{example}

\subsection{Quasi-coherent sheaves on algebraic stacks}

We refer to \cite[\href{https://stacks.math.columbia.edu/tag/06WU}{Tag 06WU}]{stacks-project} for the abelian category of quasi-coherent sheaves $\mathrm{QCoh}(\mc{Y})$ for an algebraic stack $\mc{Y}$. It has the following properties:
\begin{enumerate}
\item for an affine quotient stack $[X/G]$, we have a canonical equivalence:
  \begin{equation*}
    \mathrm{QCoh}([X/G])\cong \mathrm{QCoh}_{G}(X).
  \end{equation*}
\item for a morphism of algebraic stacks $f:\mc{X}\to \mc{Y}$ the pull-back $f^{*}$ preserves quasi-coherent sheaves and induces a functor
\begin{equation*}
  f^{*}:\mathrm{QCoh}(\mc{Y})\to \mathrm{QCoh}(\mc{X}).
\end{equation*}
Moreover, for a morphism $[f/\psi]:[X/G]\to [Y/H]$ of affine quotient stacks, the pull-back functor $[f/\psi]^{*}$ is represented by $(f,\psi)^{*}$.
\item when $f$ is quasi-compact and quasi-separated, it induces a functor
\begin{equation*}
  f_{*}:\mathrm{QCoh}(\mc{X})\to \mathrm{QCoh}(\mc{Y})
\end{equation*}
which is a right adjoint to $f^{*}$, following Proposition 103.11.1 of \cite[\href{https://stacks.math.columbia.edu/tag/070A}{Tag 070A}]{stacks-project}.  A similar construction will produce higher direct image functors $R^{i}f_{*}:\mathrm{QCoh}(\mc{X})\to \mathrm{QCoh}(\mc{Y})$ for all integers $i\geq 0$.
\item if $\mc{X}$ is locally Noetherian, i.e. there is a smooth atlas of $\mc{X}$ which is a locally Noetherian scheme, we can define the abelian category of coherent sheaves $\mathrm{Coh}(\mc{X})$ which is a full subcategory of $\mathrm{QCoh}(\mc{X})$, and $f^{*}$ induces a pull-back functor for the category of coherent sheaves if $f$ is a morphism of locally Noetherian algebraic stacks, which we also denote as $f^{*}$. Moreover, we have a canonical equivalence:
  \begin{equation*}
    \mathrm{Coh}([X/G])\cong  \mathrm{Coh}_{G}(X),
  \end{equation*}
if $R[X]$ is Noetherian.
\item The abelian category of quasi-coherent sheaves satisfies the fppf descent in the sense of flat-fppf sites (see \cite[\href{https://stacks.math.columbia.edu/tag/08MZ}{Tag 08MZ}]{stacks-project} for the introduction).
\end{enumerate}

The quasi-coherent sheaves over algebraic stacks satisfy the following flat base change theorem:
\begin{theorem}[Flat base change theorem, Lemma 103.4.1 of {\cite[\href{https://stacks.math.columbia.edu/tag/076W}{Tag 076W}]{stacks-project}}, Lemma 103.7.2 and Lemma 103.7.3 of {\cite[\href{https://stacks.math.columbia.edu/tag/0760}{Tag 0760}]{stacks-project}}] Given a flat morphism $f:\mc{X}\to \mc{Y}$ of algebraic stacks, the pullback functor
  \begin{equation*}
    f^{*}:\mathrm{QCoh}(\mc{Y})\to \mathrm{QCoh}(\mc{X})
  \end{equation*}
  is an exact functor. Moreover, a Cartesian diagram of algebraic stacks:
  \begin{equation*}
    \begin{tikzcd}
      \mc{X}_{1}\ar{r}{f'}\ar{d}{g'} & \mc{X}_{2}\ar{d}{g} \\
      \mc{Y}_{1}\ar{r}{f} & \mc{Y}_{2},
    \end{tikzcd}
  \end{equation*}
 such that $g$ is quasi-compact and quasi-separated and $f$ is flat induces canonical equivalence of natural transformations:
  \begin{equation*}
    f^{*}R^{i}g_{*}\cong R^{i}g_{*}'f'^{*}:\mathrm{QCoh}(\mc{X}_{2})\to\mathrm{QCoh}(\mc{Y}_{1}) 
  \end{equation*}
  for all non-negative integers $i$.
\end{theorem}

\subsection{Universally good morphisms}
\label{universal}
  \begin{definition}[Universally good morphisms]Let $f:\mc{X}\to \mc{Y}$ be a morphism of algebraic stacks. We say that $f$ is a universally good morphism if $f$ is locally of finite type, quasi-compact, quasi-separated, with an affine diagonal, and for any Cartesian diagram of algebraic stacks
    \begin{equation}
      \label{eqn2}
      \begin{tikzcd}
        \mc{X}'\ar{d}{f'}\ar{r} & \mc{X} \ar{d}{f}\\
        \mc{Y}'\ar{r}{g} & \mc{Y},
      \end{tikzcd}
    \end{equation}
    \begin{enumerate}
    \item  the functor
      \begin{equation*}
      f_{*}:QCoh(\mc{X})\to QCoh(\mc{Y}).
    \end{equation*}
    is exact;
  \item if $\mc{Y}'$ is locally Noetherian,  then $f_{*}$ maps coherent sheaves to coherent sheaves;
    \item the canonical morphism $\mc{O}_{\mc{Y}'}\to f_{*}\mc{O}_{\mc{X}'}$ is an isomorphism,
    \end{enumerate}
  \end{definition}
  \begin{remark}
    Obviously, the terminology ``universally good'' is from Alper's ``good moduli space'' in \cite{MR3237451}.
  \end{remark}
  By the fppf descent of quasi-coherent sheaves (resp. coherent sheaves) over algebraic stacks, we have
  \begin{lemma}Universally good morphisms are stable under base change and compositions. Moreover, given a Cartesian diagram of algebraic stacks in \cref{eqn2} such that $g$ is a fppf atlas of $\mc{Y}$, $f$ is universally good if and only if $f'$ is universally good.

    Particularly, assuming a morphism of algebraic stacks $f:\mc{X}\to \mc{Y}$ is locally of finite type, quasi-compact, quasi-separated and with an affine diagonal, $f$ is universally good if and only if for any Cartesian diagram of algebraic stacks in \cref{eqn2} such that $\mc{Y}'\cong \mathrm{Spec}(R)$, the functor
    \begin{equation*}
      f_{*}':\mathrm{QCoh}(\mc{X}')\to R-\mathrm{Mod}, \quad F\to \Gamma(\mc{X}',F)
    \end{equation*}
    is exact, maps coherent sheaves to finite generated $R$-modules if $R$ is Noetherian and the canonical morphism $R\to \Gamma(\mc{X}',\mc{O}_{\mc{X}'})$ is an isomorphism.
  \end{lemma}
  \begin{example}
    \label{exp8}
    Given an $R$-affine group $G$, we say that $BG$ is linearly reductive if the morphism $BG\to \mathrm{Spec}(R)$ is universally good. Particularly, $\mb{G}_{m,R}$ and $\mu_{l,R}$ are all linearly reductive groups by \cref{ex1}
\end{example}

\begin{example}
  Let $R[X]$ be an $\mb{Z}/l\mb{Z}$-graded $R$-algebra which is finite generated as a $R[X]_{0}$-module. Then the quotient stack morphism
  \begin{equation*}
    \psi_{l,X}:[X/\mu_{l,R}]\to \mathrm{Spec}(R[X]_{0})
  \end{equation*}
  is universally good by  \cref{ex3}.

  Particularly,  we consider the Cartesian diagram of algebraic stacks:
    \begin{equation*}
      \begin{tikzcd}[column sep=2cm]
        \left[\mb{A}^{1}/\mu_{l}\right]\ar{r}{\left[id/\mu_{l}\right]} \ar{d}{\left[x^{l}/t^{l}\right]}  & \left[\mb{A}^{1}/\mb{G}_{m}\right] \ar{d}{\left[x^{l}/t^{l}\right]}\\
         \mb{A}^{1}\ar{r}   & \left[\mb{A}^{1}/\mb{G}_{m}\right].
      \end{tikzcd}
    \end{equation*}
    The morphism
        \begin{equation*}
    [x^{l}/t^{l}]:[\mb{A}^{1}/\mu_{l}]\to \mb{A}^{1}  
  \end{equation*}
  is universally good by the above argument, and thus the morphism
   \begin{equation*}
    \theta^{l}:\Theta\to \Theta
  \end{equation*}
  is also universally good.
  \end{example}

  \begin{lemma}
    \label{lem2}
    Given a universally good morphism $f:\mc{X}\to \mc{Y}$ of algebraic stacks, the higher direct image vanishes for any $\mc{F}\in QCoh(\mc{X})$, the higher direct image vanishes:
    \begin{equation*}
      R^{i}f_{*}(\mc{F})=0, \quad i>0.
    \end{equation*}
  \end{lemma}
  \begin{proof}
    As the higher direct image is fppf local, we can assume that $\mc{Y}$ is an affine scheme. As $f$ is quasi-compact, we have a smooth cover $u: U\to \mc{Y}$ where $U$ is an affine scheme. Then for any integer $n>0$, the $n$-Cartesian product $U_{n}:=\underbrace{U\times_{\mc{X}}\times \cdots \times_{\mc{X}}U}_{n \text{ times}}$ is also an affine scheme, as the diagonal $\Delta_{\mc{X}}:\mc{X}\to \mc{X}\times \mc{X}$ is affine. The projection morphisms $u_{n}:U_{n}\to \mc{X}$ are also affine morphisms We denote $u_{n}:U_{n}\to \mc{X}$ as all $U_{n+1}$ are affine. Thus we have the long exact sequence of Cech complexes in $\mathrm{QCoh}(\mc{Y})$
    \begin{equation*}
      0\to \mc{F}\to u_{1*}u_{1}^{*}\mc{F}\to u_{2*}u_{2}^{*}\mc{F}\to \cdots.
    \end{equation*}
    By the Grothendieck spectral sequence and the fact that all $U_{n}$ and $u_{n}$ are affine, $R^{i}f_{*}\mc{F}$ is the cohomology of the following complex
    \begin{equation*}
       f_{*}u_{1*}u_{1}^{*}\mc{F}\to f_{*}u_{2*}u_{2}^{*}\mc{F}\to \cdots
     \end{equation*}
     and the higher cohomology vanishes because $f_{*}$ is an exact functor.
   \end{proof}
   \subsection{Derived category of (quasi)-coherent sheaves on algebraic stacks}
   \label{der}
   We refer to \cite[\href{https://stacks.math.columbia.edu/tag/07B5}{Tag 07B5}]{stacks-project} for the derived category of (quasi)-coherent sheaves on algebraic stacks. Given an algebraic stack $\mc{X}$, we denote $D_{qcoh}(\mc{O}_{\mc{X}})$ (resp. $D^{+}_{qcoh}(\mc{O}_{\mc{X}})$, $D^{-}_{qcoh}(\mc{O}_{\mc{X}})$, $D^{b}_{qcoh}(\mc{O}_{\mc{X}})$) as the derived category of quasi-coherent sheaves (resp. with cohomology bounded from below, above and both sides) on $\mc{X}$. We denote $D_{Coh}(\mc{O}_{\mc{X}})$ (resp. $D^{+}_{coh}(\mc{O}_{\mc{X}})$, $D^{-}_{coh}(\mc{O}_{\mc{X}})$, $D^{b}_{coh}(\mc{O}_{\mc{X}})$) as the full subcategory of $D_{qcoh}(\mc{O}_{\mc{X}})$ (resp. $D^{+}_{qcoh}(\mc{O}_{\mc{X}})$, $D^{-}_{qcoh}(\mc{O}_{\mc{X}})$, $D^{b}_{qcoh}(\mc{O}_{\mc{X}})$) with coherent cohomologies.

   Given a morphism $f:\mc{X}\to \mc{Y}$, the pull-back functor of quasi-coherent sheaves induce the derived functors
   \begin{equation*}
     Lf^{*}:D_{qcoh}(\mc{O}_{\mc{Y}})\to D_{qcoh}(\mc{O}_{\mc{X}}),
   \end{equation*}
   which maps complexes with bounded from above (resp. coherent) cohomologies to complexes with bounded from above (resp. coherent) cohomologies. Moreover, it maps complexes with bounded from below (resp. bounded) cohomologies to complexes with bounded from below (resp. bounded) cohomologies if $f$ is flat.

   Moreover, when $f:\mc{X}\to \mc{Y}$ is quasi-compact and quasi-separated, then the push-forward functor induces the derived functor
   \begin{equation*}
     Rf_{*}:D_{qcoh}^{+}(\mc{O}_{\mc{X}})\to D_{qcoh}^{+}(\mc{O}_{\mc{Y}}).
   \end{equation*}
   If $f$ is universally good, then $Rf_{*}$ maps complexes with bounded cohomologies (resp. coherent cohomologies) to complexes with bounded cohomologies (resp. coherent cohomologies).

   Like the flat base change theorem for quasi-coherent sheaves, we also have the following theorem for the derived category of quasi-coherent sheaves:
   \begin{theorem}[Proposition 2.3.2 of \cite{gaitsgory2019study}]
    A  Cartesian diagram of algebraic stacks:
  \begin{equation*}
    \begin{tikzcd}
      \mc{X}_{1}\ar{r}{f'}\ar{d}{g'} & \mc{X}_{2}\ar{d}{g} \\
      \mc{Y}_{1}\ar{r}{f} & \mc{Y}_{2},
    \end{tikzcd}
  \end{equation*}
 such that $g$ is quasi-compact and quasi-separated and $f$ is flat induces canonical equivalence of natural transformations:
  \begin{equation*}
    Lf^{*}Rg_{*}\cong Rg_{*}'Lf'^{*}:D^{+}_{qcoh}(\mc{X}_{2})\to D^{+}_{qcoh}(\mc{Y}_{1}) 
  \end{equation*}
\end{theorem}

\section{Semir-orthogonal decomposition of the root stacks}
\label{sec4}
In this section, we give a constructive proof of the semi-orthogonal decomposition of the derived category of bounded-(from-below) (quasi)-coherent sheaves on the root stack.

\subsection{Semi-orthogonal decomposition of the $\Theta_{R}$}
We first give a semi-orthogonal decomposition for the derived category of (quasi)-coherent sheaves of 
\begin{equation*}
\Theta_{R}:=\Theta\times_{\mathrm{Spec}(\mb{Z})}\mathrm{Spec}(R).  
\end{equation*}

We recall a lemma about effective Cartier divisor:
\begin{lemma}{\cite[\href{https://stacks.math.columbia.edu/tag/0B4A}{Tag 0B4A}]{stacks-project}}
  \label{lem3}
  Let $f:X\to Z$ be an effective Cartier divisor such that $Z$ is an algebraic stack. Then the right adjoint functor of
  \begin{equation*}
  Rf_{*}:D^{+}_{qcoh}(X)\to D^{+}_{qcoh}(Z) 
  \end{equation*}
  which we denote as $Lf^{!}$ is $Lf^{*}(-\otimes^{L}\mc{O}(X))[-1]$. Moreover, the ad junction formula induces a canonical triangle of derived functors
  \begin{equation*}
     L_{f}^{*}R_{f*}\to id \to L_{f}^{!}R_{f*}.
  \end{equation*}
\end{lemma}

We consider the $\mb{G}_{m,R}$-equivariant morphism:
\begin{equation*}
  i_{R}:B\mb{G}_{m,R}\to \Theta_{R}, \quad R[X]\to R, \quad x\to 0.
\end{equation*}
and the following commutative diagram
\begin{equation}
  \label{cd1}
  \begin{tikzcd}
    B\mb{G}_{m,R}\ar{d}{Bt^{l}_{R}}\ar{r}{i_{R}} & \Theta_{R} \ar{d}{\theta_{R}^{l}} \\
    B\mb{G}_{m,R}\ar{r}{i_{R}} & \Theta_{R}
  \end{tikzcd}
\end{equation}
and denote the following derived functors:
\begin{equation*}
  \triangleright_{m,R}:=RBi_{R*}(-\otimes_{R} \mc{L}_{R}^{m})LBt_{R}^{l*}, \quad \triangleleft_{m,R}:=RBt_{R*}^{l}(-\otimes_{R} \mc{L}_{R}^{m})Li_{R}^{*}.
\end{equation*}
  If $l>1$, by \cref{ex1} and \cref{ex3}, we have
  \begin{equation}
    \label{eq1}
    R\theta_{R*}^{l}\mc{O}_{\Theta_{l,R}}\cong \mc{O}_{\Theta_{l,R}},\quad  RBt^{l}_{R*}\mc{L}_{R}^{i}\cong  \mc{L}_{R}^{i/l}.
  \end{equation}
\begin{lemma}
  \label{key}
  The derived functor $\triangleleft_{i+1,R}[-1]$ is the right adjoint functor of $\triangleright_{i,R}$ and $R\theta_{R*}^{l}$ is the right-adjoint functor of $L\theta_{R}^{l*}$. Moreover, we have
  \begin{equation}
    \label{eq2}
    R\theta_{R*}^{l}L\theta_{R}^{l*}\cong id. 
  \end{equation}
  \begin{equation}
    \label{eq3}
    \triangleleft_{-j,R}\circ \triangleright_{i,R}\cong -\otimes (\mc{L}_{R}^{(i-j)/l}\oplus \mc{L}_{R}^{(i-j+1)/l}[-1])
\end{equation}
\begin{equation}
  \label{eq4}
  R\theta_{R*}^{l}\circ\triangleright_{i,R}\cong Ri_{R*}\circ (-\otimes\mc(L)_{R}^{i/l}), \quad   \triangleleft_{i,R}\circ L\theta_{R}^{l*}\cong (-\otimes\mc{L}_{R}^{i/l})Li_{R}^{*}.
\end{equation}
\begin{equation}
  \triangleright_{i+l,R}\cong \triangleright_{i,R}\circ (-\otimes_{R}\mc{L}_{R}), \quad \triangleright_{i+l,R}\cong (-\otimes_{R}\mc{L}_{R})\circ \triangleright_{i,R}.
\end{equation}
\end{lemma}
\begin{proof}
  \label{eq5}
  The equation \cref{eq1} follows from \cref{ex1} and \cref{ex3}. The equations \cref{eq2}, \cref{eq3},\cref{eq4},\cref{eq5} follows from the projection formula, \cref{lem3} and \cref{eq1}.
\end{proof}

\begin{theorem}
  \label{thm1}
  There are functors:
    \begin{equation*}
      \overline{\tau_{m,n,R}}:D^{+}_{qcoh}(\Theta)\to D^{+}_{qcoh}(\Theta), \quad 0\leq n\leq m\leq l-1
    \end{equation*}

  which maps complexes with bounded cohomologies (resp. coherent cohomologies) to complexes with bounded cohomologies (resp. coherent cohomologies) and satisfies the following properties:
  \begin{enumerate}
  \item we have
    \begin{equation*}
    \overline{\tau_{n,n,R}}\cong \otimes \mc{L}_{\mb{A}_{R}^{1}}^{n},\quad \overline{\tau_{0,l-1,R}}\cong L\theta_{R}^{l*}R\theta_{R*}^{l}
    \end{equation*}
  \item for any $0\leq n<m \leq l-1$, we have canonical triangles:
    \begin{align*}
      \overline{\tau_{n,m-1,R}}\to \overline{\tau_{n,m,R}}\to \bigoplus_{i=0}^{m} \triangleright_{m-i,R} \triangleleft_{i,R}\\
      \overline{\tau_{n,m,R}}\to \overline{\tau_{n+1,m,R}}\to \bigoplus_{i=n+1}^{l-1} \triangleright_{i-1-n,R}\triangleleft_{i,R}.
    \end{align*}.
      \item All those functors map complexes with bounded cohomologies (resp. coherent cohomologies) to complexes with bounded cohomologies (resp. coherent cohomologies).
  \end{enumerate}

\end{theorem}

\begin{proof}
  We consider the equivariant morphisms in \cref{ext3}
  \begin{align*}
    \alpha_{1}:(\mb{G}^{m}\times \mu_{l},\alpha_{l})\to (\mb{G}_{m},\mb{A}^{1}),  & \quad  ((t,\mu),(x,y))\to (t,x),\\
    \alpha_{2}:(\mb{G}^{m}\times \mu_{l},\alpha_{l})\to (\mb{G}_{m},\mb{A}^{1}), &  \quad  ((t,\mu),(x,y))\to (t\mu,y).
  \end{align*}
  which induces morphisms of quotient stacks from $[\alpha_{l}/\mb{G}_{m}\times \mu_{l}]$ to $\Theta$ and we abuse the notation to still denote them as  $\alpha_{1},\alpha_{2}$ respectively.  Then we have the Cartesian diagram of affine quotient stacks:
  \begin{equation*}
    \begin{tikzcd}
      \left[\alpha_{l}/\mb{G}^{m}\times \mu_{l}\right]\ar{r}{\alpha_{1}} \ar{d}{\alpha_{2}}& \Theta \ar{d}{\theta^{l}} \\
      \Theta \ar{r}{\theta^{l}} & \Theta.
    \end{tikzcd}
  \end{equation*}
  Such that the morphisms $\alpha_{1}$ and $\alpha_{2}$ are both flat, universally cohomologically affine and universally cohomologically proper. The diagonal $\Delta_{\theta_{l}}:\Theta\to \alpha_{l}$ is represented by $\Delta_{t^{l},x^{l}}$ in \cref{ext3} and \cref{ext4}. We denote $\alpha_{1,R}:=\alpha_{1}\otimes R$  and $\alpha_{2,R}:=\alpha_{2}\otimes R$.

  For
  $$x\in D^{b}_{coh}([\alpha_{l,R}/\mb{G}^{m}_{R}\times \mu_{l,R}]),$$
  we consider the Fourier-Mukai transform:
  \begin{equation*}
   \bar{x}: R_{\alpha_{2,R}*}(x\otimes^{L}L\alpha_{1,R}^{*}(-)):D^{+}_{qcoh}(\Theta_{R})\to D^{+}_{qcoh}(\Theta_{R}).
 \end{equation*}
 which maps complexes with bounded or coherent cohomologies to complexes with bounded or coherent cohomologies.

 Now we define $\tau_{m,n,R}:=\tau_{m,n}\otimes R$, where $\tau_{m,n}$ is defined in \cref{ext4}. By the flat base change theorem, we have $\overline{\tau_{0,l-1,R}}\cong L\theta^{l*}R\theta^{l}_{*}$ and
  \begin{equation*}
    R_{\alpha_{2,R}*}(R<m,n>\otimes^{L}L\alpha_{1,R}^{*}(-))\cong \triangleright_{n-m,R}\triangleleft_{m,R}.
  \end{equation*}
  Thus \cref{thm1} follows from \cref{ext4}.
\end{proof}

\begin{corollary}
  If $l>1$, the following functors are fully faithful:
  \begin{align*}
    L\theta^{*}_{l}:D_{qcoh}^{+}(\Theta_{R})\to D_{qcoh}^{+}(\Theta_{R}), \quad \triangleright_{i,R}:D_{qcoh}^{+}(B\mb{G}_{m,R})\to D_{qcoh}^{+}(\Theta_{R}).
  \end{align*}
  Moreover, we denote $D_{l,R}^{i}:=\triangleright_{i,R}D_{qcoh}^{+}(B\mb{G}_{m,R})$. Then
  \begin{equation*}
    D_{l,R}^{i}\cong D_{l,R}^{i+l}
  \end{equation*}
and  for any $0\leq i\leq l-1$, we have the semi-orthogonal decomposition
\begin{equation}
  \label{eq6}
    <D_{l,R}^{i-l+1},\cdots, D_{l,R}^{-1}, L\theta^{*}_{l}D^{+}_{qcoh}(\Theta_{R}), D_{l,R}^{0},\cdots, D_{l,R}^{i-1}>
  \end{equation}
  which is a fully subcategory of $D^{+}_{qcoh}(\Theta)$. Moreover, similar arguments also hold for for derived category of complexes with bounded or coherent cohomologies.
\end{corollary}

\subsection{The semi-orthogonal decomposition for root stacks}
Now we give a proof for the semi-orthogonal decomposition of the derived category of (below)-bounded (quasi)-coherent sheaves of root stacks.

Let $\mc{X}$ be an algebraic stack. Given a line bundle $\mc{L}$ over $\mc{X}$, it induces a canonical morphism from $\mc{X}\to B\mb{G}_{m}$ such that $\mc{L}$ is the pull-back of $\mc{L}_{\mb{Z}}$. Given an integer $l>1$, we denote
\begin{equation*}
  \mc{X}_{\mc{L},l}:=\mc{X}\times_{B{\mb{G}_{m}},Bt^{l}}B\mb{G}_{m}.
\end{equation*}
and denote $\mc{L}_{l}\in Pic(\mc{X}_{\mc{L},l})$ as the pull-back of $\mc{L}_{\mb{Z}}$ along the second projection $\mc{X}_{\mc{L},l}\to B\mb{G}_{m}$.

Let
\begin{equation*}
  r_{\mb{Z}}:\mc{L}_{\mb{A}_{\mb{Z}}^{1}}\to \mc{O}_{\Theta}
\end{equation*}
be the canonical on $\Theta$ which maps $1$ to $x$. Given a line bundle $\mc{L}$ with a co-section
\begin{equation*}
  r:\mc{L}\to \mc{O}_{\mc{X}}
\end{equation*}
over $\mc{X}$, it induces a canonical morphism from $\mc{X}$ to $\Theta$ such that $r$ is the pull-back of $r_{\mb{Z}}$. We define the $l$-th root stack of $r$ as
\begin{equation*}
  \mc{X}_{r,l}:=\mc{X}\times_{\Theta, \theta^{l}}\Theta.
\end{equation*}

Now we assume that $r$ is generated by an effective divisor $\mc{D}$, i.e.
\begin{equation*}
  r=(\mc{O}_{\mc{X}}(-\mc{D})\subset \mc{O}_{\mc{X}}).
\end{equation*}
 It induces a Cartesian diagram
\begin{equation*}
  \begin{tikzcd}
    \mc{D}\ar{r} \ar{d}{i_{\mc{X}}} & B\mb{G}_{m}\ar{d}{i_{\mb{Z}}}\\
    \mc{X}\ar{r}{r} & \Theta.
  \end{tikzcd}
\end{equation*}
such that the morphism from $\mc{D}$ to $B\mb{G}_{m}$ is induced by the line bundle $\mc{L}_{\mc{D}}:=\mc{O}_{\mc{X}}(\mc{-D})|_{\mc{D}}$. We denote
\begin{equation*}
  \mc{X}_{\mc{D},l}:=\mc{X}_{r,l}, \quad \mc{D}_{l}:=\mc{D}_{\mc{L}_{\mc{D}},l}
\end{equation*}
 and consider the pull-back of \cref{cd1} for $R=\mb{Z}$ along $r$, and have the following Cartesian diagram
\begin{equation*}
  \begin{tikzcd}
    \mc{D}_{l}\ar{r}{i_{\mc{X}_{\mc{D},l}}} \ar{d}{Bt_{\mc{D}}^{l}} & \mc{X}_{\mc{D},l}\ar{d}{\theta_{\mc{X}}^{l}}\\
    \mc{D}\ar{r}{i_{\mc{X}}} & \mc{X}.
  \end{tikzcd}
\end{equation*}
We notice that $Bt^{l}_{\mc{D}}$ and $\theta_{\mc{X}}^{l}$ are all universally good morphisms by the base change. For any integer $m$, we define the following functors
\begin{align*}
  \triangleleft_{m,\mc{X}}:=RBt_{\mc{D}*}^{l}\circ (-\otimes \mc{L}_{\mc{D},l}^{m}) \circ Li_{\mc{X}_{\mc{D},l}}^{*}:D^{+}_{qcoh}(\mc{X}_{\mc{D},l})\to D_{qcoh}^{+}(\mc{D}),\\
  \triangleright_{m,\mc{X}}:=Ri_{\mc{X}_{\mc{D},l}*}\circ (-\otimes \mc{L}_{D,l}^{m})\circ LBt_{\mc{D}}^{l*}:D^{+}_{qcoh}(\mc{D})\to D_{qcoh}^{+}(\mc{X}_{\mc{D},l}).
\end{align*}

By the fppf descent of quasi-coherent sheaves, \cref{lem3} and \cref{key}, we have the following lemmas
\begin{lemma}
  \label{main1}
  If $l>1$, we have
  \begin{equation}
    \label{eq11}
    R\theta_{\mc{X}*}^{l}\mc{O}_{\mc{X}_{\mc{D},l}}\cong \mc{O}_{\mc{X}},\quad  RBt^{l}_{\mc{D}*}(\mc{L}_{\mc{D},l}^{i})\cong \mc{L}_{\mc{D}}^{i/l}
  \end{equation}
    The derived functor $\triangleleft_{i+1,\mc{X}}[-1]$ is the right adjoint functor of $\triangleright_{i,\mc{X}}$ and $R\theta_{\mc{X}*}^{l}$ is the right-adjoint functor of $L\theta_{\mc{X}}^{l*}$. Moreover, we have
  \begin{equation}
    \label{eq12}
    R\theta_{\mc{X}*}^{l}L\theta_{\mc{X}}^{l*}\cong id. 
  \end{equation}
  \begin{equation}
    \label{eq13}
    \triangleleft_{-j,\mc{X}}\circ \triangleright_{i,\mc{X}}\cong -\otimes (\mc{L}_{\mc{D}}^{(i-j)/l}\oplus \mc{L}_{\mc{D}}^{(i-j+1)/l}[-1])
\end{equation}
\begin{equation}
  \label{eq14}
  R\theta_{\mc{X}*}^{l}\circ\triangleright_{i,\mc{X}}\cong Ri_{\mc{X}*}\circ (-\otimes\mc{L}_{\mc{D}}^{i/l}), \quad   \triangleleft_{i,\mc{X}}\circ L\theta_{\mc{X}}^{l*}\cong (-\otimes\mc{L}_{\mc{D}}^{i/l})Li_{\mc{X}}^{*}.
\end{equation}
\begin{equation}
  \label{eq15}
  \triangleright_{i+l,\mc{X}}\cong \triangleright_{i,\mc{X}}\circ (-\otimes\mc{L}_{\mc{D}}), \quad \triangleright_{i+l,\mc{X}}\cong (-\otimes\mc{L}_{\mc{D}})\circ \triangleright_{i,\mc{X}}.
\end{equation}
\end{lemma}

\begin{theorem}
  \label{main2}
Given $0\leq n\leq m\leq n-1$, we consider
\begin{equation*}
  \tau_{n,m,\mc{X}}:=r^{*}\tau_{n,m}\in D^{b}_{coh}(\mc{X}_{\mc{D},l}\times_{\theta_{\mc{X}}^{l},\mc{X},\theta_{\mc{X}}^{l}}\mc{X}_{\mc{D},l}).
\end{equation*}
and define
\begin{equation*}
\overline{\tau}_{n,m,\mc{X}}:D^{+}_{qcoh}(\mc{X}_{\mc{D},l})\to D^{+}_{qcoh}(\mc{X}_{\mc{D},l})
\end{equation*}
as the Fourier-Mukai transform generated by the Fourier-Mukai kernels $\tau_{n,m,\mc{X}}$. Then
which maps complexes with bounded cohomologies (resp. coherent cohomologies) to complexes with bounded cohomologies (resp. coherent cohomologies) and satisfy the following properties:
\begin{enumerate}
  \item we have
    \begin{equation*}
    \overline{\tau_{n,n,\mc{X}}}\cong \otimes \mc{O}_{\mc{X}_{\mc{D},l}}(-\mc{D}_{l})^{n},\quad \overline{\tau_{0,l-1,\mc{X}}}\cong L\theta_{\mc{X}}^{l*}R\theta_{\mc{X}*}^{l}
    \end{equation*}
  \item for any $0\leq n<m \leq l-1$, we have canonical triangles:
    \begin{align*}
      \overline{\tau_{n,m-1,\mc{X}}}\to \overline{\tau_{n,m,\mc{X}}}\to \bigoplus_{i=0}^{m} \triangleright_{m-i,\mc{X}} \triangleleft_{i,\mc{X}}\\
      \overline{\tau_{n,m,\mc{X}}}\to \overline{\tau_{n+1,m,\mc{X}}}\to \bigoplus_{i=n+1}^{l-1} \triangleright_{i-1-n,\mc{X}}\triangleleft_{i,\mc{X}}.
    \end{align*}.
      \item All those functors map complexes with bounded cohomologies (resp. coherent cohomologies) to complexes with bounded cohomologies (resp. coherent cohomologies).
  \end{enumerate}
\end{theorem}
By \cref{main1} and \cref{main2}, we have the following theorem
\begin{theorem}
  \label{main3}
  If $l>1$, the following functors are fully faithful:
  \begin{align*}
    L\theta^{l*}_{\mc{X}}:D_{qcoh}^{+}(\mc{X})\to D_{qcoh}^{+}(\mc{X}_{\mc{D},l}), \\
    \triangleright_{i,\mc{X}}:D_{qcoh}^{+}(\mc{D})\to D_{qcoh}^{+}(\mc{X}_{\mc{D},l}).
  \end{align*}
  Moreover, we denote $D_{l,\mc{D}}^{i}:= \triangleright_{i,\mc{X}}D_{qcoh}^{+}(\mc{D})$. Then
  \begin{equation*}
    D_{l,\mc{D}}^{i}\cong D_{l,\mc{D}}^{i+l}
  \end{equation*}
and  for any $0\leq i\leq l-1$, we have the semi-orthogonal decomposition
\begin{equation}
  D^{+}_{qcoh}(\mc{X}_{\mc{D},l}):=<D_{l,\mc{D}}^{i-l+1},\cdots, D_{l,\mc{D}}^{-1}, L\theta^{*}_{l}D^{+}_{qcoh}(\mc{X}), D_{l,\mc{D}}^{0},\cdots, D_{l,\mc{D}}^{i-1}>.
\end{equation}
Similar arguments also hold for complexes with bounded or coherent cohomologies.
\end{theorem}

\bibliography{resolution.bib}
\bibliographystyle{plain}
\vspace{5mm}
Kavli Institute for the Physics and 
Mathematics of the Universe (WPI), University of Tokyo,
5-1-5 Kashiwanoha, Kashiwa, 277-8583, Japan.

\textit{E-mail address}: yu.zhao@ipmu.jp
\end{document}